\documentclass[11pt, fleqn, reqno]{amsart}
\usepackage{amsmath,amsfonts,amssymb,amsthm}
\usepackage{epsfig}

\usepackage{bbm}
\usepackage{bm}
\usepackage{enumerate}
\usepackage{amstext}
\usepackage{mathrsfs}
\usepackage{xspace}
\usepackage{amsfonts}
\usepackage{amsmath}
\usepackage{amssymb}
\usepackage{amstext}
\usepackage{amsthm}       %proof environment :)
\usepackage{xspace}
\usepackage[active]{srcltx}
\usepackage{cite}

\newcommand{\I}{\mathcal I}
\newcommand{\ind}{\{1,\dots,m\}}

\newcommand{\F}{\mathcal F}
\newcommand{\FF}{\mathbb F}
\newcommand{\R}{\mathbb R}

\newcommand{\N}{\mathbb N}

\newcommand{\T}{\mathbb T}

\newcommand{\sol}{\mathfrak{S}}

\newcommand{\eps}{\varepsilon}
\newcommand{\ucv}{\rightrightarrows}
\newcommand{\ff}{\mathbf f}
\newcommand{\gb}{\mathbf g}
\newcommand{\pphi}{\bm\phi}
\newcommand{\uu}{\mathbf u}
\newcommand{\vv}{\mathbf v} 
\newcommand{\ww}{\mathbf w}
\newcommand{\aaa}{\mathbf a}

\newcommand{\EE}{\mathbb E}
\newcommand{\PP}{\mathbb{P}}
\newcommand{\HH}{\mathbb{H}}

\newcommand{\1}{\mathbbm1}

\newcommand{\Bor}{\mathscr{B}}

\newcommand{\e}{\textrm{\rm e}}

\newcommand{\Id}{\D{Id}}

\newcommand{\Mis}{\mathfrak{M}}

\newcommand{\tagliato}{$\kern-5 mm -$}
\newcommand{\tagliat}{$\kern-4 mm -$}

\newcommand{\D}[1]{\mbox{\rm #1}}
\newcommand{\dd}{\D{d}}

\newcommand{\om}{\omega}

\newcommand{\cyl}{(0,+\infty)\times M}

\newcommand{\vLip}{\big(\D{Lip}(M)\big)^m}

\newcommand{\curves}{\D{C}\big(\R_{+};M\big)}
\newcommand{\vC}{\big(\D{C}( M )\big)^m}
\newcommand{\vCuno}{\big(\D{C}^1( M )\big)^m}
\newcommand{\vCc}{\big(\D{C}_c( TM )\big)^m}

\newtheorem{teorema}{Theorem}
\newtheorem*{teorema*}{Theorem}
\newtheorem{prop}[teorema]{Proposition}
\newtheorem{lemma}[teorema]{Lemma}
\newtheorem{definition}[teorema]{Definition}

\newtheorem{guess}[teorema]{Remark}
\newtheorem{example}[teorema]{Example}

\newenvironment{oss}{\begin{guess} \begin{rm}}{\end{rm} \end{guess}}

\begin{document}

\title[Convergence of the solutions of discounted HJ systems]{Convergence of the solutions\\ of discounted Hamilton--Jacobi systems}

\thanks{\rm Work supported by ANR-07-BLAN-0361-02 KAM faible \&
ANR-12-BS01-0020 WKBHJ}
\author{Andrea Davini \and Maxime Zavidovique}
\address{Dip. di Matematica, {Sapienza} Universit\`a di Roma,
P.le Aldo Moro 2, 00185 Roma, Italy}
\email{davini@mat.uniroma1.it}
\address{
IMJ-PRG (projet Analyse Alg\' ebrique), UPMC,  
4, place Jussieu, Case 247, 75252 Paris Cedex 5, France}
\email{zavidovique@math.jussieu.fr} \keywords{asymptotic behavior of solutions, Mather measures, weak KAM Theory, viscosity solutions, optimal control}
\subjclass[2010]{35B40, 37J50, 49L25.}

\date{February 24, 2017}
\begin{abstract} We consider a weakly coupled system of discounted Hamilton--Jacobi equations set on a closed Riemannian manifold.  We prove that the corresponding solutions  converge to a specific solution of the limit system as the discount factor goes to $0$. 
The analysis is based on a generalization of the theory 
 of Mather minimizing measures for Hamilton--Jacobi systems and on suitable random representation  formulae for the discounted solutions.
\end{abstract}
%
%\dedicatory{In memory of John Mather}
%
\maketitle

\section*{Introduction}

In this paper, we are interested in the asymptotic behavior, as $\lambda\to 0^+$, of the solutions of the following system of weakly coupled Hamilton--Jacobi equations
\begin{equation*}
\sum_{j=1}^m b_{ij} u_j+\lambda u_i+H_i(x,D  u_i)=c\qquad\hbox{in $M$}
\end{equation*}
for $i\in\ind$, where $M$ is a compact, connected Riemannian manifold without boundary,  $c$ is a real number,   
$H_1,\dots,H_m$ are continuous function on $T^*M$, convex and coercive in the gradient variable, and $B= (b_{ij})$ is an $m\times m$ irreducible and weakly diagonally dominant matrix, see Section \ref{sys} for the precise assumptions.  The solution $\uu = (u_1,\cdots , u_m)^T: M\to \R^m$ is assumed to be continuous and to solve the above system in the viscosity sense. The sign and degeneracy condition assumed on the coefficients of $B$
amounts to requiring that $-B$ is the generator of a semigroup of stochastic matrices.

It is convenient to restate the system in the following vectorial form
\begin{equation}\label{intro discounted system}
(B+\lambda \Id)\uu +\HH(x, D  \uu) = c\1\qquad\hbox{in $M$},
\end{equation}
where we have used the notations   $\HH(x, D  \uu) = \big(H_1(x, D  u_1), \cdots ,H_m(x, D  u_m) \big) ^T$ and $\1=(1\cdots,1)^T\in\R^m$. The conditions assumed on $B$ imply, in particular, that $B\1=0$. 

When $\lambda=0$, there is a unique value $c$ for which \eqref{intro discounted system} admits solutions, hereafter denoted by $c(\HH)$ and termed {\em critical}. Furthermore, the solutions of the {\em critical system} 
\begin{equation}\label{intro critical system}
B\uu +\HH(x, D  \uu) = c(\HH)\1\qquad\hbox{in $M$}
\end{equation}
are not unique, not even up to addition of vectors of the form $a\1$, in general. 

When $\lambda>0$, on the other hand, the system \eqref{intro discounted system} satisfies a comparison principle, yielding the existence of a unique continuous solution $\uu^{\lambda,c}:M\to\R^m$ for every fixed $c\in \R$. Moreover, the solutions $\{\uu^{\lambda,c}\mid\lambda>0\}$ are equi--Lipschitz.
% and the functions $\{\lambda\uu^{\lambda,c}\mid\lambda>0\}$ are equi--bounded. 
The peculiarity of the discounted system \eqref{intro discounted system} when   $c:=c(\HH)$ relies on the fact that the corresponding solutions $\uu^\lambda:=\uu^{\lambda,c(\HH)}$ are also equi--bounded. 
By Ascoli--Arzel\`a Theorem and by the stability of the notion of viscosity solution, we infer that  they uniformly converge, {\em along subsequences} as $\lambda$ goes to $0$, to viscosity solutions of the critical system \eqref{intro critical system}. Since the solutions of the critical system are not unique, it is not clear at this level that the limits of the $\uu^\lambda$ along different subsequences yield the same critical solution.\smallskip

In this paper, we address this question.  The main theorem we will establish is the following:

\begin{teorema}\label{intro main teo}
Let $\uu^\lambda$ be the solution of system \eqref{intro discounted system} with $c:=c(\HH)$ and $\lambda>0$. The functions $\uu^\lambda$ uniformly converge as $\lambda\to 0^+$ to a single solution $\uu^0$ of the critical system \eqref{intro critical system}.
\end{teorema}

We will characterize $\uu^0$ in terms of a generalized notion of Mather minimizing measure for HJ systems.\smallskip

Notice that the relationship between $\uu^{\lambda}$ and $\uu^{\lambda,c}$ when $c$ varies is rather straightforward: it is easily verified that 
$\uu^{\lambda,c} = \uu^{\lambda} +  \frac{c-c(\HH)}{\lambda}\1$. 
%In particular, the critical value $c(\HH)$ can be also characterized as the unique constant $c$ for which the solutions $\{\uu^{\lambda,c}\mid\lambda>0\}$ are %equi--bounded. 
%
As a consequence, we derive from Theorem \ref{intro main teo} the following fact:

\begin{teorema}\label{intro teo consequence}
Let $\uu^{\lambda,c}$ be the solution of system \eqref{intro discounted system} with $\lambda>0$. Then, as $\lambda\to 0^+$, the 
functions $\lambda \uu^{\lambda,c}$ uniformly converge in $M$ to the constant vector $\big(c-c(\HH)\big)\1$ and the functions $\widehat\uu^{\lambda,c}:=\uu^{\lambda,c}-\min_i\min_x u^{\lambda,c}_i\1$ uniformly converge to $\uu^0-\min_i\min_x u^0_i\1$ in $M$.
\end{teorema}

Theorem \ref{intro teo consequence} for $c=0$ can be restated by saying that the {\em ergodic approximation} selects a specific critical solution in the limit. The ergodic approximation 
is a classical technique introduced in \cite{LPV}  for the case of a single equation (i.e. with $m=1$ and $B=0$). Since then, it has been 
extended and applied to many different 
settings, including the case of weakly coupled systems of Hamilton--Jacobi equations, see \cite{leyetal,mitaketran}. This technique is typically employed to show the existence and uniqueness of the critical value $c(\HH)$ and the existence of a solution of the corresponding critical problem. 
The latter is usually obtained by renormalizing the discounted solutions so to produce a family of equi--bounded and equi--Lipschitz functions satisfying suitable perturbed discounted problems (for instance, the family $\{\widehat\uu^{\lambda,0}\mid\lambda>0\}$ in the case of HJ systems) and by taking limits, {\em along subsequences} as $\lambda\to 0^{+}$, of these renormalized functions.  The fact that the limit is unique has been recently established in \cite{DFIZ1} for the case of a single equation  by using tools and results issued from weak KAM Theory. 
This  selection principle has been subsequently generalized in different directions, see  \cite{AEEI, DFIZ2, GoMiTr-discount, IsMiTr-discount1, IsMiTr-discount2, MiTr-discount}, testifying the interest for the issue.\smallskip 

The extension of the selection principle  to HJ systems provided in the present work is based on a generalization of the theory
of Mather minimizing measures, which is new in this setting and 
enriches the frame of analogies with weak KAM theory for scalar Eikonal equations.
This stream of research was initiated  in \cite{leyetal} with the proof of the long--time
convergence of the solutions to evolutive HJ systems, under
hypotheses close to  \cite{NR}. Other outputs in this vein can be
found in a series of works including \cite{Ng,MT2}. The links
with weak KAM theory were further made precise by the authors
of the present paper in \cite{DavZav14} where, by purely using PDE tools and viscosity solution techniques, an appropriate  notion of Aubry set for systems was
given and some relevant properties of it were generalized from the
scalar case. A dynamical and variational point of view of the matter, integrating
the PDE methods, was later brought in by \cite{SMT14,ISZ}.  This angle allowed the authors to detect the
stochastic character of the problem, displayed by the random
switching nature of the dynamics and by the role of an adapted action functional. Representation
formulae for viscosity (sub)solutions of the critical systems and a
cycle characterization of the Aubry set were derived.

The random frame introduced in \cite{SMT14} and subsequently developed in \cite{DSZ} is the starting point of our analysis. 
It is exploited to provide suitable random representation  formulae for the solutions of both the critical and the discounted system. 
A point that is crucial to our purposes consists in showing the existence of admissible minimizing curves in such formulae. 
This is done by making use of the results proved in \cite{DSZ} and by adapting  the construction therein employed to the discounted system case.\smallskip
 
 The paper is organized as follows: in Section \ref{sez preliminaries} we fix notations and the standing assumptions, and we provide some preliminary results on the critical and discounted systems. In Section \ref{sez representation formulae} we present the random frame in which our analysis takes place and we prove suitable random representation formulae for the solutions of the critical and discounted systems. In Section \ref{sez Mather theory} we generalize the theory of Mather minimizing measures to the case of HJ systems. Section \ref{sez main teo} contains the proof of Theorem \ref{intro main teo}.

\smallskip\indent{\textsc{Acknowledgements. $-$}}
This research was initiated in May 2015, while the first author was visiting, as Professeur Invit\'e, the Institut de Math\'ematiques de Jussieu,  
Universit\'e Pierre et Marie Curie (Paris), that he gratefully acknowledges for the financial support and hospitality.  
\numberwithin{teorema}{section}
\numberwithin{equation}{section}

\begin{section}{Preliminaries}\label{sez preliminaries}

\begin{subsection}{Notations}\label{sez notation}
In this work, we will denote by $M$ the $N$--dimensional flat torus $\T^N$, where $N$ is an integer number. This is done to simplify the notation and to be consistent with the references we will use. We remark however that our results and proofs keep holding, {\em mutatis mutandis}, whenever $M$ is a 
compact connected Riemannian manifold without boundary.  The associated Riemannian distance on $M$ will  be denoted by $d$. 
We denote by $TM$ the tangent bundle and by $(x,v)$ a point of $TM$, with $x\in M$ and $v\in T_x M=\R^N$. 
In the same way, a point of the cotangent bundle $T^*M$ will be denoted by $(x,p)$, with $x\in M$ and 
$p\in T_x^* M$ a linear form on the vector space $T_x M$. The latter will be identified with the vector $p\in\R^N$ such that 
\[
 p(v)=\langle p,v\rangle\qquad\hbox{for all $v\in T_xM=\R^N$,}
\]
where $\langle\,\cdot\;, \cdot\,\rangle$ denotes the Euclidean scalar product in $\R^N$. The fibers $T_xM$ and $T_x^*M$ are endowed with the Euclidean norm $|\cdot|$, for every $x\in M$.
%
%We will denote by $B_R(x_0)$ and $B_R$ the closed balls in $ M $ of
%radius $R$ centered at $x_0$ and $0$, respectively.

With the symbols $\N$ and  $\R_+$ we will refer to the set of positive integer numbers and 
nonnegative  real numbers, respectively.  We say
that a property holds {\em almost everywhere} ($a.e.$ for short)
in a subset $E$ of  $M$ (respectively, of $\R$) if it holds up to a {\em negligible} subset of $E$, i.e. a
subset of zero $N$--dimensional (resp., $1$--dimensional) Lebesgue measure.
%
%\indent By modulus we mean a nondecreasing function from $\R_+$ to
%$\R_+$, vanishing and continuous at $0$. 
% 
% We will say that $(\rho_n)_n$ is a sequence of {\em standard mollifiers} if \ $\rho_n(x):=n^N\rho(nx)$ in $\R^N$\quad for each $n\in\N$, where $\rho$ is a smooth, non--negative function on $\R^N$, supported in $B_1$ and such that its integral over $\R^N$ is equal to $1$.\smallskip\par  

Given a continuous function $u$ on $M$ and a point $x_0\in M$, we will denote by $D^-u(x_0)$ and $D^+u(x_0)$ the set of {\em subdifferential}  and {\em superdifferential} of $u$ at $x_0$, respectively.
%
%we will call  {\em subtangent} (respectively, {\em supertangent}) of $u$ at $x_0$ a function $\phi$ of class $ C^1$ in a neighborhood $U$ of $x_0$ such that $\phi(x_0)= u(x_0)$ and $\phi(x)\leqslant u(x)$  for every $x\in U$ (resp., $\geqslant$). Its gradient $D\phi(x_0)$ will be called a {\em subdifferential}  (resp. {\em superdifferential}) of $u$ at $x_0$, respectively. The set of sub and superdifferentials of $u$ at $x_0$ will be denoted $D^-u(x_0)$ and 
%$D^+u(x_0)$, respectively. We recall that $u$ is
%differentiable at $x_0$ if and only if $D^+u(x_0)$ and $D^-u(x_0)$ are
%both nonempty. In this instance, $D^+u(x_0)=D^-u(x_0)=\{Du(x_0)\}$,
%where $Du(x_0)$ denotes the differential of $u$ at $x_0$. We refer the
%reader to \cite{CaSi00} for the 
%proofs.
%
When $u$ is locally Lipschitz in $ M $, we will denote by $\partial^c u(x_0)$ the set of {\em Clarke's generalized gradient} of $u$ at $x_0$, see \cite{Cl} for a detailed presentation of the subject.
%
%
%$\partial^*u(x_0)$ the set of 
%{\em reachable gradients} of $u$ at $x_0$, that is the set
%\[
%\partial^* u(x_0)=\{\lim_n D u(x_n)\,:\,\hbox{$u$ is differentiable at $x_n$, $x_n\to x_0$}\,\},
%\]
%while the {\em Clarke's generalized gradient} $\partial^c u(x_0)$ is the closed convex hull of $\partial^* u(x_0)$. 
%The set $\partial_c u(x_0)$ contains both $D^+u(x_0)$ and $D^-u(x_0)$, in particular $Du(x_0)\in \partial_c u(x_0)$ at any differentiability point $x_0$ of $u$. 

\indent 
We will denote by $\|g\|_\infty$ the usual $L^\infty$--norm of $g$, where the latter is a   
measurable real function defined on $ M $. 
% We will write
% $g_n\ucv g$ in $ M $ to mean that the sequence of
% functions $(g_n)_n$ uniformly converges to $g$ in $ M $. 
We will denote by $\big(\D{C}( M )\big)^m$ the Banach space of continuous functions $\uu=(u_1,\dots,u_m)^T$ 
from $ M $ to $\R^m$, endowed with the norm 
\[
 \|\uu\|_\infty=\max_{1\leqslant i\leqslant m}\|u_i\|_\infty,\qquad\hbox{$\uu\in\big(\D{C}( M )\big)^m$}.
\]
We will write $\uu^n\ucv \uu$ in $ M $ to mean that $\|\uu^n-\uu\|_\infty\to 0$. A function $\uu\in \big(\D{C}( M )\big)^m$ will be termed Lipschitz continuous if each of its components is $\kappa$--Lipschitz continuous, for some $\kappa>0$. Such a constant $\kappa$ will be called a {\em Lipschitz constant} for $\uu$. The space of all such functions will be denoted by $\big(\D{Lip}( M )\big)^m$. 
\smallskip\par

We will denote by $\1=(1,\cdots,1)^T$ the vector of $\R^m$ having all components equal to 1, where the upper--script symbol $T$ stands for the transpose. We consider the following partial relations between elements  $\mathbf{a},\mathbf{b}\in\R^m$:
$\aaa \leqslant \mathbf{b}$  if $a_i\leqslant b_i$ (resp., $<$) for every $i\in\ind$. Given 
two functions $\uu,\vv: M \to\R^m$, we will write $\uu \leqslant \vv$  in $ M $ (respectively, $<$) to mean that $\uu(x)\leqslant \vv(x)$ \big(resp., $\uu(x)<\vv(x)$\big)  {for every $x\in M $}. 
\smallskip\par
\end{subsection}

\begin{subsection}{Weakly coupled systems}\label{sys}
Throughout the paper, we will assume the Hamiltonians $H_i$ to be continuous functions on $ T^*M$ satisfying, for every $i \in \ind$,
\smallskip

\begin{itemize}
    \item[(H1)] (convexity) \ $p\mapsto H_i(x,p)\qquad\hbox{is  convex on $\R^N$ for any
    $x\in  M $;}$ \smallskip\\
     \item[(H2)] (coercivity) \ there exist two coercive functions $\alpha,\beta:\R_+\to\R$ such that
     \[
     \qquad \alpha(|p|)\leqslant H_i(x,p)\leqslant \beta(|p|)\quad \hbox{for every $(x,p)\in T^*M$}.
     \]
    \end{itemize}
For our analysis, it will be convenient and non restrictive, see Section \ref{sez representation formulae}, to reinforce this coercivity condition in favor of the following:
\begin{itemize}
     \item[(H2$'$)] (superlinearity) \ there exist two superlinear functions $\alpha,\beta:\R_+\to\R$ such that
     \[
     \qquad \alpha(|p|)\leqslant H_i(x,p)\leqslant \beta(|p|)\quad \hbox{for every $(x,p)\in T^*M$}.
     \]
    \end{itemize}
We recall that a function $f:\R_+\to\R$ is termed {\em coercive} if $f(h)\to +\infty$ as $h\to +\infty$, while it is termed {\em superlinear} if  $f(h)/h\to +\infty$ as $h\to +\infty$. 

In the sequel, we will denote by $\partial_p H_i(x,p)$ the set of subdifferentials at $p$ of the function $p\mapsto H_i(x,p)$ in the sense of convex analysis. 
We recall that, due to conditions (H1)--(H2), the function  $H_i(x,\cdot)$ is locally Lipschitz in $T_x^*M$, with a local Lipschitz constant that can be chosen independent of $x\in M$. In particular, the sets $\{\partial_p H_i(x,p)\mid x\in M,\,|p|\leqslant R\,\}$ are uniformly bounded for fixed $R>0$.\smallskip

The {\em coupling matrix} $B=(b_{ij})$ has dimensions  $m\times m$
and satisfies
\begin{itemize}
 \item[(B1)] $b_{ij}\leqslant 0\ \hbox{for $j\not=i$,}\quad\quad \sum_{j=1}^m b_{ij}= 0$;\medskip
 \item[(B2)] $B$ is irreducible, i.e. for every subset $\I\subsetneq\{1,\dots,m\}$ there exist $i\in\I$ and $j\not\in\I$ such that $b_{ij}\not=0$.\smallskip
\end{itemize}

For $\lambda\geqslant 0$ and $c\in\R$, we consider the following weakly coupled system of Hamilton--Jacobi equations
\begin{equation}\label{wcoupled system}
(B+\lambda\Id)\uu +\HH(x, D \uu) = c\1\qquad\hbox{in $M$},
\end{equation}
where we have adopted the notation $ \HH(x, D \uu) = \big(H_1(x, D u_1), \cdots ,H_m(x, D u_m) \big) ^T$. 

Let $\uu\in\left(\D{C}( M)\right)^m$. We will say that $\uu$ is a {\em viscosity subsolution} of \eqref{wcoupled system} if the following inequality holds for every $(x,i)\in M\times\ind$
\begin{equation*}
H_i(x,p)+\big((B(x)+\lambda\Id)\uu(x)\big)_i\leqslant c\quad\hbox{for every $p\in D^+ u_i(x)$.}
\end{equation*}
We will say that $\uu$ is a {\em viscosity supersolution} of \eqref{wcoupled system} if the following inequality holds for every $(x,i)\in M\times\ind$
\begin{equation*}
H_i(x,p)+\big((B(x)+\lambda\Id)\uu(x)\big)_i\geqslant c\quad\hbox{for every $p\in D^- u_i(x)$.}
\end{equation*}
We will say that $\uu$ is a {\em viscosity solution} if it is both a sub and a supersolution. In the sequel, solutions, subsolutions and supersolutions will be always meant in the viscosity sense, hence the adjective {\em viscosity} will be  omitted. 

When $\lambda=0$, there exists a unique value $c$ for which the system \eqref{wcoupled system} admits solutions, hereafter denoted by $c(\HH)$ and termed {\em critical}. In fact, $c(\HH)$ can be also characterized as 
\begin{equation}\label{eq characterization c(H)}
c(\HH)=\min\left\{c\in\R\,|\,\hbox{system \eqref{wcoupled system} with $\lambda=0$ admits subsolutions}   \right\},
\end{equation}
see \cite{DavZav14} for a detailed analysis. 

We recall from \cite{DavZav14} the following result, that will be crucial for our analysis:

\begin{prop}\label{prop a priori Lip}
Let $\uu=(u_1,\dots,u_m)^T \in \big(\D{C}( M )\big)^m$ be a  subsolution of
\eqref{wcoupled system} with $\lambda=0$ and $c\in\R$.
Then there exist  constants $C_c $ and $\kappa_c $, only depending on $c$, on the Hamiltonians $H_1,\dots, H_{m}$ and on the coupling matrix $B$, such that
\begin{itemize}
 \item[\em (i)]  \quad $\|u_i-u_j\|_\infty\leqslant C_c \ \quad\qquad\hbox{for every $i,\,j\in\ind$;}$\medskip
 \item[\em (ii)] \quad $\uu$ is $\kappa_c$--Lipschitz continuous in $ M $.
\end{itemize}
\end{prop}

We proceed presenting some basic facts about the discounted system, i.e. system \eqref{wcoupled system} when $\lambda>0$. 
The following existence and uniqueness result depends on the fact that the matrix $B+\lambda \Id$ is non degenerate as soon as $\lambda >0$.

\begin{prop}\label{comparison}
Let $\lambda >0$ and $c\in \R$. Let $\vv, \uu\in \big(\D{C}( M )\big)^m$ be respectively a subsolution and a supersolution to \eqref{wcoupled system}, then $\vv\leqslant \uu$. In particular, there exists a unique solution $\uu^{\lambda, c}$ in $\vC$.
\end{prop}

\begin{proof}
The first assertion is a consequence of Proposition 2.8 in \cite{DavZav14}, while the second follows via a standard application of Perron's method. 
\end{proof}

As already mentioned in the Introduction, the relationship between those solutions when $c$ varies is given by \ $\uu^{\lambda, c} = \uu^{\lambda,c'} +\frac{c-c'}{\lambda}\1$. In particular, it follows that as $\lambda \to 0^+$, the family $\uu^{\lambda, c}$ may be bounded at most for one value $c$. 

We now explain why this is the case for $c=c(\HH)$. 

\begin{prop}\label{prop u lambda}
Let us denote by $\uu^\lambda$ the unique solution in $\vC$ of \eqref{wcoupled system} with $c=c(\HH)$ and $\lambda>0$.  
Then the functions  $\{\uu^\lambda\,|\,\lambda>0\,\}$ are equi--Lipschitz and equi--bounded. In particular, $\|\lambda \uu^\lambda\|_\infty\to 0$ as $\lambda\to 0^+$. 
\end{prop}
\begin{proof}
Let $\uu\in\vC$ be a solution of \eqref{wcoupled system} with $c=c(\HH)$ and $\lambda=0$.  By taking $A>0$ big enough, it follows that $\overline u:=\uu+A\1$ takes only positive values and $\underline\uu:= \uu-A\1$ takes only negative values. Therefore, $\overline\uu$ and $\underline\uu$ are respectively a super and a subsolution of \eqref{wcoupled system} with $c=c(\HH)$ for any parameter $\lambda >0$. By Proposition \ref{comparison} we infer that $\underline\uu \leqslant \uu^\lambda \leqslant \overline \uu$ in $M$ for all $\lambda>0$, thus proving the asserted equi--bounded character of the $\{\uu^\lambda\mid\lambda>0\}$. 

Let us now prove that $\uu^\lambda$ is Lipschitz and its Lipschitz constant can be chosen independent of $\lambda>0$.  
Let us set $b=\max_{i\in\ind}\max_{x\in M}H_i(x,0)$. The function 
$\ww\equiv-\1\big(b-c(H)\big)/\lambda$ is obviously a subsolution of \eqref{wcoupled system} with $c=c(\HH)$.
By Proposition \ref{comparison}, 
we must have $\lambda \uu^\lambda\geqslant \big(-b+c(H)\big)\1$ in $M$, hence
\[
B\uu^\lambda +\HH(x, D \uu^\lambda) = -\lambda \uu^\lambda + c(\HH)\1 \leqslant b\1 \qquad\hbox{in $M$}
\]
in the viscosity sense. According to Proposition \ref{prop a priori Lip} we conclude that $\uu^\lambda$ is $\kappa$--Lipschitz, where the constant $\kappa$ only depends on the constant $b$, on the Hamiltonians $H_1,\dots, H_{m}$ and on the coupling matrix $B$.
\end{proof}

\begin{oss}\label{oss beta}
Note that $b:=\max_{x\in M}H_i(x,0) \geqslant c(H)$. This readily follows from the characterization of $c(\HH)$ given in \eqref{eq characterization c(H)} after noticing 
that the null function is a subsolution of \eqref{wcoupled system} with $\lambda=0$ and $c=b$. \\
\end{oss}
\end{subsection}

\end{section}

\section{Random representation formulae for solutions}\label{sez representation formulae}

In this section, we will establish suitable representation formulae for the solution of the following system
\begin{equation}\label{system discounted critical}
(B+\lambda\Id)\uu +\HH(x, D \uu) = c(\HH)\1\qquad\hbox{in $M$}
\end{equation}
when either $\lambda>0$ or $\lambda=0$. This will be done by adopting the random frame introduced in \cite{DSZ} and by 
adapting the strategy therein employed to the case at issue. In the sequel, we shall refer to the system \eqref{system discounted critical} and  its corresponding (sub, super) solutions as {\em discounted} when $\lambda>0$, {\em critical} when $\lambda=0$. 

To implement this program, we need to assume that the Hamiltonians satisfy the stronger growth assumption (H2$'$). We want to explain 
here why this is not restrictive for our analysis. According to the proof of Proposition \ref{prop u lambda}, the discounted 
solutions $\uu^\lambda$ satisfy 
\begin{equation*}\label{system discounted critical}
B\uu^\lambda +\HH(x, D \uu^\lambda) \leqslant b\1\qquad\hbox{in $M$}
\end{equation*}
in the viscosity sense with $b:=\max_i\max_{x}H_i(x,0)$. In view of Remark \ref{oss beta}, this is also true for the (sub-)solutions of the critical system. Therefore all these functions are $\kappa$--Lipschitz continuous, with $\kappa=\kappa_b$ chosen according to Proposition \ref{prop a priori Lip}. We can therefore modify each Hamiltonian $H_i$ outside the compact set  $K:=\{(x,p)\in T^*M\ |\ |p|\leqslant \kappa\}$ to obtain a new Hamiltonian $\widetilde H_i$ which is still continuous and convex, and satisfies the stronger growth condition (H2$'$).  Since $H_i\equiv \widetilde H_i$ on $K$ for each $i\in\ind$, it is easily seen that $c(\HH)=c(\widetilde \HH)$ and the solutions of the corresponding critical and discounted systems are the same.

In the remainder of the paper, we will therefore assume each Hamiltonian $H_i$ to be convex and superlinear in $p$, i.e. hypotheses (H1) and   (H2$'$) will be in force. This allows us to introduce the associated Lagrangian $L_i: TM\to \R$ defined as follows:
\begin{equation}\label{def L}
L_i(x,v):=\sup_{p\in\R^N}\left\{\langle p, v\rangle - H_i(x,p)\right\}\qquad\hbox{for every $(x,v)\in TM$}.
\end{equation}
As well known, $L_i$ satisfies properties analogous to (H1)--(H2$'$). By definition of $L_i$ we derive 
\[
H_i(x,p)+L_i(x,v) \geqslant \langle p,q\rangle\qquad\hbox{for all $(x,p)\in T^*M$ and $(x,v)\in TM$,}
\]
which is known as {\em Fenchel's inequality}.

\subsection{Random frame}\label{sez prob}
We briefly recall  the random frame in which our
analysis takes place, see \cite{DSZ} for more details.  
We take as
sample space $\Omega$ the space of paths $\omega:\R_+\to\ind$\ \ 
 that are right--continuous and possess
left--hand limits (known in
literature as {\em c\`{a}dl\`{a}g paths}, a French acronym for {\em
continu \`a droite, limite \`a gauche}, see Billingsley's book
\cite{Bill99} for a detailed treatment of the topic). By
c\`{a}dl\`{a}g property  and the fact that  the range of
$\omega\in\Omega$ is finite, the points of discontinuity of  any
such path are isolated and consequently finite in compact intervals
of $\R_+$ and countable (possibly finite) in the whole of $\R_+$. We
call them {\em jump times} of $\omega$.

The space  $\Omega$ is endowed  with a distance, named after  {\em
Skorohod}, see \cite{Bill99},  which turns it into a Polish space.
We denote by $\F$ the corresponding Borel  $\sigma$--algebra and,
for every $t\geqslant 0$, by $\pi_t:\Omega\to\ind$ the map that
evaluates each $\omega$ at $t$, i.e. $\pi_t(\omega)=\omega(t)$ for
every $\omega\in\Omega$.
It is known that $\F$ is the minimal $\sigma$--algebra that makes
all the functions $\pi_t$ measurable, i.e.
 $\pi_t^{-1}(i)\in\F$ for every $i\in\ind$ and $t\geqslant 0$. 
% For every  $t \geqslant 0$, we will 
% denote by $\F_t$ the minimal $\sigma$--algebra that makes all the functions $\{\pi_s\,:\,s\leqslant t\,\}$ measurable. Then
%$\left\{\F_t\right\}_{t\geqslant 0}$ is a {\em filtration of $\F$},
% i.e. $\F_s\subseteq \F_t$ for every $0\leqslant s <t$ and $\cup_{t\geqslant 0}\,\F_t=\F$.
% It is in addition {\em right--continuous} in the sense that  $\F_t= \cap_{s >t} \F_s$ for any $t$.
%Note that $\F_0$ comprises a finite number of sets, namely $\Omega$,
%$\varnothing$, $\Omega_i:=\{\omega\in\Omega\,:\,\omega(0)=i\,\}$
%for every $i \in \ind$, and unions of such sets.
%
% 
% Let $\mu$ be a probability measure on $(\Omega,\F)$. Given $E \in \F$, we define the {\em restriction } of $\mu$ to $E$
% as\ \ $\mu \restr E (F) = \mu (E \cap F) \quad\hbox{for any $F \in
% \F$}.$
% The probability $\mu$ conditioned to the event $E\in\F$ is defined as
% \begin{eqnarray*}
% \mu(F\,|\,E):=\frac{\mu(F\cap E)}{\mu(E)}\qquad\hbox{for every $F\in\F$},
% \end{eqnarray*}
% where we agree that $\mu(F\,|\,E)=0$ whenever $\mu(E)=0$.

Let us now fix an $m\times m$ matrix $B$ satisfying assumption
(B1)--(B2). We record that $\e^{-tB}$ is a stochastic matrix for every $t\geqslant 0$, namely
a matrix with nonnegative entries and with each row summing to 1. We endow $\Omega$ of a probability measure $\PP$ defined on
the $\sigma$--algebra $\F$ in such a way that the right--continuous
process $\left(\pi_t\right)_{t\geqslant 0}$ is a {\em Markov chain
with generator matrix $-B$}, i.e. it satisfies the Markov property
\begin{eqnarray}\label{Markov property}
\PP\big(\omega(t_k)=i_k\,|\,\omega(t_1)=i_1,\dots,\omega(t_{k-1})=i_{k-1}\,\big)
= \left(\e^{-B(t_k-t_{k-1})}\right)_{i_{k-1}i_k}
\end{eqnarray}
for all times $0\leqslant t_1<t_2<\dots<t_k$, states
$i_1,\dots,i_k\in\ind$ and $k\in\N$. We will denote by $\PP_i$ the
probability measure $\PP$ conditioned to the event $\Omega_i:=\{\omega\in\Omega\mid \omega(0)=i\}$ and write $\EE_i$ for the
corresponding expectation operators.  It is easily seen that the Markov property
\eqref{Markov property} holds with $\PP_i$ in place of $\PP$, for every $i\in\ind$.

%
% A converse construction is also possible. Given  $ i \in \ind$, and the
% corresponding vector $\mathbf e_i$ of the canonical basis of  $
% \R^m$,  we associate to any  $\mathcal C(t_1, \cdots, t_k;i_1,
% \cdots, i_k)$ the quantity
% \[\left (\mathbf e_i \,\e^{-t_1 B} \right )_{i_1} \, \prod_{j=2}^{k}
% \left (\e^{-(t_l-t_{l-1})B} \right )_{i_{j-1} \,i_j}.\] By Kolmogorov
% Extension theorem, there exists an unique probability measure,
% denoted by $\PP_i$,  which extends the previous function to the whole
% of $\F$, see  \cite[Theorem 14.36]{K1}.

In the sequel, we will call
{\em random variable} a  map $X:(\Omega,
\F)\to\big(\FF,\Bor(\FF)\big)$, where $\FF$ is a Polish space and
$\Bor(\FF)$  its Borel $\sigma$--algebra, satisfying
$X^{-1}(A)\in\F$ for every $A\in\Bor(\FF)$.
% 
% Given a probability measure $\mu$  on $(\Omega,\F)$,  we denote by
% $X_\#\mu$ the {\em push--forward} of $\mu$ through the map $X$, i.e.
% the probability measure on $\Bor(\FF)$ defined as
% \[\left({X}_\#\mu\right)(A):=\mu\left(\{\omega\in\Omega\,:\,X(\omega)\in A\,\}\right)\quad\hbox{for
% every $A \in \Bor(\FF)$.}\]
Let us denote by
$\curves$ the Polish space of continuous paths taking values in $M$, endowed with a metric that
induces the topology of local uniform convergence in $\R_+$.

We call  {\em admissible curve}  a random variable $\gamma:\Omega\to
\curves$ such that
\begin{itemize}
 \item[\em (i)] it  is uniformly (in $\omega$) locally  (in $t$) absolutely continuous,
 i.e. given any bounded interval $I$ and $\eps >0$,  there is $\delta_\eps >0$
 such that
 \begin{equation}\label{h1 control}
  \sum_j (b_j-a_j) < \delta_\eps  \; \Rightarrow  \; \sum_j d\big(\gamma(b_j,\omega), \gamma(a_j,\omega)\big)
  < \eps
   \end{equation}
   for any finite family  $\{(a_j,b_j)\}$ of pairwise disjoint intervals
   contained in I and for any $\omega \in \Omega$;
 \item[\em (ii)]  it is {\em nonanticipating}, i.e. for any $t \geqslant 0$
 \begin{equation}\label{h2 control}
 \omega_1\equiv\omega_2\ \hbox{in $[0,t]$}\quad\Rightarrow \gamma(\cdot,\omega_1)\equiv \gamma(\cdot,\omega_2)\ \hbox{in $[0,t]$}.\medskip
 \end{equation}
\end{itemize}
We will say that $\gamma$ is an admissible curve starting at $y\in M$ when $\gamma(0,\omega)=y$ for every $\omega\in\Omega$.\smallskip

Given an admissible curve $\gamma:\Omega\to\curves$ and $\omega\in\Omega$, we will denote by $\|\dot\gamma(\cdot,\omega)\|_{\infty}$ the $L^\infty$--norm of the derivative of the curve $\gamma(\cdot,\omega)$.\smallskip
 
We record for later use the following Dynkin's formula, see \cite[Theorem 4.7]{DSZ} for a proof:

\begin{teorema}\label{Dynkin's formula}
Let $\gb:\R_+\times M\to\R^m$ be a locally Lipschitz function and  $\gamma$ an admissible curve. Then, for every index $i\in\ind$, we have
\begin{eqnarray}\label{eq Dynkin's formula}
    \frac{\dd}{\dd t}\EE_i  \big[g_{\omega(t)}\big(t,\gamma(t,\omega)\big)\big ]_{\mbox{\Large
$|$}_{t=s}} = \EE_i \Big
[-\big(B\gb\big)_{\omega(s)}\big(s,\gamma(s,\omega)\big)  + \frac{\dd}{\dd t}
g_{\omega(s)}\big(t,\gamma(t,\omega)\big)_{\mbox{\Large $|$}_{t=s}}\Big ]
\end{eqnarray}
 for
a.e. $s \in \R_+$.\\
\end{teorema}

%
%\begin{oss}
%Notice that $\gamma(0,\omega)$  is constant on $\Omega_i$, for every $i \in \ind$, by condition (ii) above with $t=0$. We refer to this value as
%  the  {\em  starting point} of $\gamma$ on $\Omega_i$.
%\end{oss}

\subsection{Representation formulae} In this section, we establish some representation formulae for solutions of the system \eqref{system discounted critical}. We begin with the critical system.

\begin{teorema}\label{teo minimal curve critical}
Let $\uu\in\vLip$ be a critical solution, namely a solution of \eqref{system discounted critical} with $\lambda=0$. Let $(y,\ell)\in M\times\ind$ and $t>0$ be fixed.
\begin{itemize}
\item[(i) ]The following holds:
\begin{equation*}   %\label{random Lax-Oleinik formula}
u_\ell(y) = \inf_{\gamma(0,\omega)=y} \EE_\ell\left[
u_{\omega(t)}\big(\gamma(t,\omega)\big)+\int_0^t
\big(L_{\omega(s)}\big(\gamma(s,\omega),-\dot\gamma(s,\omega)\big)+c(\HH)\big)\,\dd s\right],
\end{equation*}
where the minimization is performed over all admissible curves $\gamma:\Omega\to\curves$ starting at $y$.\medskip
\item[(ii)] There exists an admissible curve $\eta:\Omega\to\curves$ starting at $y$ for which such a minimum is attained.  Moreover, for every $\omega\in \Omega$, the following holds:
\begin{equation}    \label{eq Lip curve}
-\dot\eta(s,\omega)\in \partial_p H_{\omega(s)}\Big(\eta(s,\omega),\partial^c u_{\omega(s)}\big(\eta(s,\omega)\big)\Big)\qquad
\hbox{for a.e. $s\in (0,t)$.}
\end{equation}
In particular, there exists a constant $k^*$, only depending on $H_1,\dots,H_m$ and $B$ such that \ 
$\|\dot\eta(\cdot,\omega)\|_\infty\leqslant k^*$ \  for every $\omega\in\Omega$. 
\end{itemize}
\end{teorema}

\begin{proof}
The assertion follows as a simple consequence of the results proved in \cite{DSZ}. It is easily seen that the function \  $\vv(t,x):=\uu(x)$ \  is a solution of the time--dependent system 
\[
\frac{\partial \vv}{\partial t}+B\vv + \HH( x,D  \vv) - c(\HH)\1=0\qquad\hbox{in $\cyl$}
\]
with initial datum $\vv(0,\cdot)=\uu$. Item (i) and the first assertion in (ii) readily follow from \cite[Theorem 6.1]{DSZ}.  Let us prove \eqref{eq Lip curve}. Fix $\omega\in\Omega$. According to Lemma 6.8 and Lemma 1.4 in \cite{DSZ}, for a.e. $s\in (0,t)$ there exists $p_s\in \partial^c u_{\omega(s)}\big(\eta(s,\omega)\big)$ such that 
\[
\langle p_s, -\dot\eta(s,\omega)\rangle= L_{\omega(s)}\big(\eta(s,\omega),-\dot\eta(s,\omega)\big)+c(\HH)+\Big(B\uu\big(\eta(s,\omega)\big)\Big)_{\omega(s)},
\]
hence, by Fenchel's duality we get \ $-\dot\eta(s,\omega)=\partial_p H_{\omega(s)}\left(\eta(s,\omega),p_s\right)$. The remainder of the statement follows from Proposition \ref{prop a priori Lip} and the fact that $\partial_p H_i(x,p)$ is bounded on compact subsets of $T^*M$ due to (H1)--(H2$'$). 
\end{proof}

Let us now consider the discounted system. 

\begin{teorema}\label{teo minimal discounted curve}
Let $\uu^\lambda\in\vLip$ be the solution of \eqref{system discounted critical} with $\lambda>0$. Let $(y,\ell)\in M\times\ind$ be fixed.
\begin{itemize}
\item[(i) ]The following holds:
\begin{equation}  \label{discounted representation}
u^\lambda_\ell(y)
=
% \inf_{\gamma(0,\omega)=x}
%\EE_i\left[ \e^{-\lambda t} u_{\om(t)}\big(\gamma(t)\big)+\int_0^{t}\e^{-\lambda s} \Big(L_{\omega(s)}\big(\gamma(s),-\dot\gamma(s)\big)+c(\HH)\Big)\,\dd s\right] = \\
\inf_{\gamma(0,\omega)=y}
\EE_\ell\left[ \int_0^{+\infty}\e^{-\lambda s} \Big(L_{\omega(s)} \big(\gamma(s,\omega),-\dot\gamma(s,\omega)\big)+c(\HH)\Big)\,\dd s\right],
\end{equation}
where the minimization is performed over all admissible curves $\gamma:\Omega\to\curves$ starting at $y$.\medskip
\item[(ii)] There exists an admissible curve $\eta^\lambda:\Omega\to\curves$ starting at $y$ for which such a minimum is attained.  Moreover, for every $\omega\in \Omega$, the following holds:
\begin{equation}  \label{eq Lip discounted curve}
-\dot\eta^\lambda(s,\omega)\in \partial_p H_{\omega(s)}\Big(\eta^\lambda(s,\omega),\partial^c u_{\omega(s)}\big(\eta^\lambda(s,\omega)\big)\Big)\qquad
\hbox{for a.e. $s\in (0,+\infty)$.}
\end{equation}
In particular, there exists a constant $k^*$, only depending on $H_1,\dots,H_m$ and $B$ such that \ 
$\|\dot\eta^\lambda(\cdot,\omega)\|_\infty\leqslant k^*$ \  for every $\omega\in\Omega$ and $\lambda>0$.\\ 
\end{itemize}
\end{teorema}
\begin{proof}
Let $\gamma:\Omega\to\curves$ be an admissible curve starting at $y$. By applying Dynkin's formula to the function $\gb(t,x):=\e^{-\lambda t}\uu^\lambda(x)$ and by integrating \eqref{eq Dynkin's formula} on $(0,+\infty)$ we get 
\[
u^\lambda_\ell(y)=\EE_{\ell}
\left[
 \int_0^{+\infty} \e^{-\lambda t}
\left(  
(B+\lambda\Id) \uu^\lambda 
\big)_{\omega(s)}
\big(\gamma(s,\omega)\big)    
\right) 
+
\langle Du^\lambda_{\omega(s)},-\dot\gamma(s,\omega)\rangle
  \right].
\]
We now make use of Fenchel's inequality together with the fact that $\uu^\lambda$ is a solution of the discounted system 
\eqref{system discounted critical}. 
Arguing as in the proof of Proposition 5.6 in \cite{DSZ} we end up with 
\begin{eqnarray}\label{inequality discounted curve}
u^\lambda_\ell(y)\leqslant \EE_\ell\left[ \int_0^{+\infty}\e^{-\lambda s} \Big(L_{\omega(s)} \big(\gamma(s,\omega),-\dot\gamma(s,\omega)\big)+c(\HH)\Big)\,\dd s\right].
\end{eqnarray}
Next, we prove that there exists an admissible curve $\eta^\lambda:\Omega\to\curves$ starting at $y$ for which \eqref{inequality discounted curve} holds with an equality. This will be obtained via a slight modification of the strategy employed in \cite{DSZ}. 
Let $\vv(t,x) = \e^{\lambda t}\uu^\lambda(x)$. It is readily verified that $\vv$ verifies the following system:
$$\frac{\partial \vv}{\partial t}+B\vv +\e^{\lambda t}\left( \HH( x,\e^{-\lambda t}D  \vv) - c(\HH)\1\right)=0\qquad\hbox{in $\cyl$}.
$$
%Let $\tilde \HH ( t,x,p)  = \e^{\lambda t} \HH( x,\e^{-\lambda t}p)$. Its Legendre transform is given by $\tilde\LLL(t,x,v) = \e^{\lambda t}\LLL(x,v)$. 
In particular $v_i $ is, for each fixed $i\in\ind$, a solution to the equation
$$\frac{\partial v_i}{\partial t} + G_i(t,x,D  v_i) = 0\qquad\hbox{in $\cyl$,}
$$
where $G_i(t,x,p) =\e^{\lambda t}  \big(H_i(x,\e^{-\lambda t} p)-c(\HH)\big) +\sum_{k=1}^m b_{ik}v_k(t,x)$.
As $v_i$ is locally Lipschitz, it is standard, see for instance Appendix A in \cite{DSZ}, that it verifies the following Lax--Oleinik formula for every $(t,y)\in (0,+\infty)\times M$:
\begin{equation}\label{eq scalar value function}
v_i(t,y) = \inf_{\gamma}v_i\big(0,\gamma(-t)\big)+\int_{-t}^0 L_{G_i}\big(t+s,\gamma(s),\dot\gamma(s)\big)\, \dd s,
\end{equation}
where $L_{G_i}$ is the Lagrangian associated to $G_i$ by duality and the infimum is taken amongst all absolutely continuous curves $\gamma : [-t,0]\to M$ such that $\gamma(0) = y$. By standard results in the Calculus of Variations, we know that this infimum is in fact a minimum. For any fixed $\tau>0$, let us denote by $\gamma_{\tau,y}: [-\tau,0]\to M$ be an absolutely continuos curve with $\gamma_{\tau,y}(0) = y$ and realizing 
 the minimum in \eqref{eq scalar value function} with $t:=\tau$. By the Dynamic Programming Principle, such a curve $\gamma_{\tau,y}$ is also a minimizer 
 of \eqref{eq scalar value function} for every $t\leqslant \tau$. 
Arguing as in the proof of Theorem \ref{teo minimal curve critical}, we get
\begin{eqnarray}\label{eq deterministic discounted curve}
\dot\gamma_{\tau,y}(s)\in\partial_p G_i\Big  (t+s,\gamma_{\tau,y}(s),\partial^c v_i\big(t+s,\gamma_{\tau,y}(s)\big)\Big)=\partial_p H_i \big(\gamma_{\tau,y}(s),\partial^c u_i(s,x)\big)
\end{eqnarray} 
for a.e. $s\in (-t,0)$. Due to the equi--Lispchitz character of the functions $\{\uu^\lambda\mid\lambda>0\}$ established in Proposition \ref{prop u lambda}, we infer that there exists a constant $\kappa^*$, independent of $(t,y)\in (0,+\infty)\times M$ and $\lambda>0$, so that \ $\|\dot\gamma_{\tau,y}\|_\infty\leqslant \kappa^*$. 
Note that $L_{G_i}(t,x,v)= \e^{\lambda t}\big(L_i(x,v)+c(\HH)-\sum_{k=1}^m b_{ik}u^\lambda_k(x)\big)$. 
It follows that
\begin{eqnarray*}
u^\lambda_i(y) = \e^{-\lambda t}u^\lambda_i\big(\gamma_{\tau,y}(-t)\big) + \int_{-t}^0 \e^{\lambda s}\Big(L_i\big(\gamma_{\tau,y}(s),\dot\gamma_{\tau,y}(s)\big)+c(\HH) - \sum_{k=1}^m b_{ik}u^\lambda_k\big(\gamma_{\tau,y}(s)\big)\Big)\,\dd s
\end{eqnarray*}
for every $t\leqslant \tau$. 
Letting $\tau\to +\infty$ and extracting a subsequence, we obtain a curve $\gamma_{i,y}  : (-\infty,0]\to M$ with $\gamma_{i,y}(0)=y$ and satisfying the previous 
equality for every $t>0$.  
%
%uch that, for all $T>0$,
%\begin{eqnarray*}
%u^\lambda_i(y) = \e^{-\lambda T}u^\lambda_i\big(\gamma_{i,y}(-T)\big) + \int_{-T}^0 \e^{\lambda s}\Big(L_i\big(\gamma_{i,y}(s),\dot\gamma_{i,y}(s)\big)+c(\HH) - \sum_{k=1}^m b_{ik}u^\lambda_k\big(\gamma_{i,y}(s)\big)\Big)\,\dd s.
%\end{eqnarray*}
By sending $t\to +\infty$, we end up with
\begin{align}\label{LO}
u^\lambda_i(y)= \int_{-\infty}^0 \e^{\lambda s}\Big(L_i\big(\gamma_{i,y}(s),\dot\gamma_{i,y}(s)\big) +c(\HH)- \sum_{k=1}^m b_{ik}u^\lambda_k\big(\gamma_{i,y}(s)\big)\Big)\,\dd s.
\end{align}
Now the proof ends exactly as in \cite{DSZ}. For every  $(y,i)\in M\times\ind $, we denote by $\Gamma(y,i)$ the set of absolutely continuous curves $\gamma:(-\infty,0]\to M$ with $\gamma(0)=y$ satisfying \eqref{LO}. The set $\Gamma(y,i)$ is nonempty, in view of the preceding discussion. Moreover, any curve in $\Gamma(y,i)$ satisfies \eqref{eq deterministic discounted curve} for a.e. $s\in (0,+\infty)$, in particular it is $\kappa^*$--Lipschitz continuous. We derive that \ $(y,i)\mapsto\Gamma(y,i)$ \ 
is compact--valued and upper semicontinuous as a set--valued map from $M\times\ind$ to $\curves$, in particular it is measurable. 
By \cite[Theorem III.8]{CV}, there exists a measurable function\ $\Xi:M\times\ind\to\curves$\ 
such that
\[
\Xi(y,i) \in \Gamma(y,i) \qquad\hbox{for every $(y,i)\in M\times\ind$.}
\]
For any fixed $\omega\in\Omega$, let  $\big(\tau_k(\omega)\big)_{k\geqslant 0}$ be the sequence of jump times of $\omega$, where $\tau_0(\omega):=0$ and 
$\tau_k(\omega)$ is the $k$--th jump time. We define inductively a sequence $\big(y_k(\omega)\big)_{k\geqslant 0}$ of points in $M$ by setting $y_0:=y$ and
\[
y_k(\omega):=\Xi\big(y_{k-1}(\omega),\omega(\tau_{k-1}(\omega)\big)\big(\tau_k(\omega)\big)\quad\hbox{for every $k\geqslant 1$}.
\]
The sought curve is given by 
\begin{eqnarray*}
\eta^\lambda(t,\omega):=\Xi\Big(y_{k}(\omega),\omega\big(\tau_{k}(\omega)\big)\Big)(-t)\qquad\hbox{if \ \ $t\in \big[\tau_{k}(\omega),\tau_{k+1}(\omega)\big)$},
\end{eqnarray*}
for every $k\geqslant 0$ and $\omega\in\Omega$. Arguing as in \cite[Section 6]{DSZ}, one can check that $\eta^\lambda$ is an admissible curve starting at $y$ for which \eqref{inequality discounted curve} holds with an equality. The fact that $\eta^\lambda$ satisfies \eqref{eq Lip discounted curve} is clear by construction in view of \eqref{eq deterministic discounted curve}.
\end{proof}
\medskip

\section{Mather measures for the critical system}\label{sez Mather theory}

In this section we generalize the notion of Mather minimizing measure to the case of the critical system, i.e. 
\begin{equation}\label{critical system}
B\uu +\HH(x, D \uu) = c(\HH)\1\qquad\hbox{in $M$}.
\end{equation}
It is not so surprising that such measures will be concentrated on the support of minimizing controls associated to solutions of \eqref{critical system}. 

We start by adapting the notion of closed measure to this setting.

\begin{definition}\label{def closed measure}
A probability measure $\mu$ on $ TM\times\ind$ will be termed closed if 
\begin{itemize}
\item[(i)] \quad $\displaystyle\int_{ TM\times\ind} |v|\ \dd \mu(x,v,i) <+\infty$;\medskip
\item[(ii)] \quad $\displaystyle\int_{ TM\times\ind } \big(B\pphi(x)\big)_i + \langle D \phi_i(x), v\rangle\  \dd \mu(x,v,i) = 0$\qquad for every $\pphi\in\vCuno$.\medskip
\end{itemize}
We will denote by $\Mis$ the set of closed measures on $TM\times\ind$.\\
\end{definition}

\begin{teorema}\label{ineq1}
The following holds:
\begin{equation}\label{eq Mather measures}
-c(\HH) = \min_{\mu\in\Mis} \int_{ TM\times\ind} L_i(x,v)\ \dd\mu(x,v,i).
\end{equation}
In particular, $\Mis$ is non empty. 
\end{teorema}
\begin{proof}
We first observe that, for every $\eps>0$, there exists a function $\ww^\eps\in\vCuno$ such that 
\begin{equation}\label{eq C1 inequality}
B\ww^\eps +\HH(x, D \ww^\eps) \leqslant \big(c(\HH)+\eps\big)\1\qquad\hbox{for every $x\in M$}.
\end{equation}
To see this, take a solution $\uu$ of \eqref{critical system} and regularize it via convolution with a standard mollifier. The above inequality follows, for a 
proper choice of the mollifier, via a well known argument based on Jensen's inequality, the convexity of the Hamiltonians and the fact that $\uu$ is Lipschitz. 

By integrating \eqref{eq C1 inequality} with respect to a measure $\mu\in \Mis$ and by using Fenchel's inequality we get:
$$
\int_{ TM\times\ind}(B\ww^\eps)_i +\langle D  w^\eps_i(x), v\rangle - L_i(x,v)\   \dd \mu(x,v,i)\leqslant c(\HH)+\eps .
$$
Since $\mu$ is closed, the left hand side is equal to $-\int_{TM\times\ind} L_i\ \dd \mu$. By letting $\eps \to 0^+$ we obtain
\[
 \int_{ TM\times\ind} L_i(x,v)\ \dd\mu(x,v,i)\geqslant -c(\HH).
\]
Let us now proceed to prove the existence of a minimizing closed measure. To this aim, take a critical solution $\uu$ and fix $(y,\ell)\in M\times\ind$. For every $k\in\N$, let 
$\eta_k:\Omega\to\curves$ be an admissible curve starting at $y$ and such that 
\begin{equation}\label{eq building Mather measure}
u_\ell(y)=\EE_\ell\left[ u_{\om(k)}\big(\eta_k(k,\omega)\big)+\int_0^{k} \Big(L_{\omega(s)}\big(\eta_k(s),-\dot\eta_k(s)\big)+c(\HH)\Big)\,\dd s\right].  
\end{equation}
We define a probability measure $\mu_k$ on $TM\times\ind$ by setting
\[
\int_{TM\times\ind} \ff\,\dd\mu_k
%\langle \ff,\mu_k\rangle
:=\frac1k\, \EE_\ell\left[ \int_0^k f_{\omega(s)}\big(\eta_k(s,\omega),-\dot\eta_k(s,\omega)\big)\,\dd s \right],
\qquad\hbox{$\ff\in\vCc$.}
\]
In view of Theorem \ref{teo minimal curve critical}, these measures have support contained in a common compact subset of $TM\times\ind$, so, up to subsequences, $(\mu_k)_k$ weakly 
converges to a probability measure $\mu$ on $TM\times\ind$. Let us show that $\mu$ is closed. 
It clearly satisfies item (i) of Definition \ref{def closed measure} since its support is compact. Let $\pphi\in\vCuno$. 
By applying Dynkin's formula to the function $\gb(t,x):=\pphi(x)$, see Theorem \ref{Dynkin's formula}, and by integrating \eqref{eq Dynkin's formula} in $(0,k)$
 we get 
\begin{align*} 
&\EE_\ell
\left[ \int_0^k  \big(B\pphi\big)_{\omega(s)}\big(s,\eta_k(s,\omega)\big)+ \langle D\phi_{\omega(s)}\big(\eta_k(s,\omega)\big),-\dot\eta_k(s,\omega)\rangle\,\dd s \right]\\
&\qquad=
 \phi_\ell(y)-\, \EE_\ell \big[\phi_{\omega(k)}\big(\eta_k(k,\omega)\big) \big] ,
\end{align*}
otherwise stated
\[
\int_{ TM\times\ind } \big(B\pphi(x)\big)_i + \langle D \phi_i(x), v\rangle\  \dd \mu_k(x,v,i) = 
\frac{ \phi_\ell(y)-\, \EE_\ell \big[\phi_{\omega(k)}\big(\eta(k,\omega)\big) \big] }{k}.
\]
By sending $k\to +\infty$ we infer that $\mu$ satisfies item (ii) in Definition \ref{def closed measure} as well. To prove that $\mu$ is minimizing, we remark that, in view of \eqref{eq building Mather measure} and the fact that the measures $(\mu_k)_k$ have equi--compact support, we have 
\begin{align*}
&\int_{ TM\times\ind} \big(L_i(x,v)+c(\HH)\big)\ \dd\mu
 =
 \lim_{k\to +\infty}  \int_{ TM\times\ind} \big(L_i(x,v)+c(\HH)\big)\ \dd\mu_k\\
&\qquad=  \lim_{k\to +\infty} \frac1k\Big(u_\ell(y)-\EE_\ell\left[ u_{\om(k)}\big(\eta_k(k,\omega)\big)\right]\Big)=0.
\end{align*}
\end{proof}

We will call {\em Mather measure} a closed probability measure on $TM\times\ind$ which minimizes \eqref{eq Mather measures}. The set of Mather measures will be denoted by $\Mis_0$ in the sequel.\medskip

\section{Convergence of the discounted solutions}\label{sez main teo}

This section is devoted to the proof of Theorem \ref{intro main teo}, namely that the solutions $(\uu^\lambda)_{\lambda>0}$ of the discounted system \eqref{system discounted critical} converge to a particular solution $\uu^0$ of the critical system \eqref{critical system} as $\lambda\to 0^+$.  

The first step consists in identifying a good candidate $\uu^0$ for the limit of the solutions $\uu^\lambda$. 
To this aim, we consider the family ${\mathcal F} $ of subsolutions  $\ww\in\vC$ of the critical system \eqref{critical system} satisfying the following condition
\begin{equation}\label{CONDITiON U0}
\int_{TM\times\ind} w_i(y)\,\dd\mu(y,v,i)\leqslant 0 \qquad\text{for every $\mu\in\Mis_0$}, 
\end{equation}
where $\Mis_0$ denotes the set of Mather measures, see Section \ref{sez Mather theory}. 

Note that, given any critical subsolution $\ww$, the function $\ww-\1\lVert \ww\rVert_\infty$  
is in ${\mathcal F} $. Therefore ${\mathcal F} $ is not empty.
\begin{lemma}\label{lemma F meno}
The family ${\mathcal F} $ is uniformly bounded from above, i.e.
$$\sup\{w_i(x)\mid \ww\in {\mathcal F} \}<+\infty\qquad\hbox{for every $(x,i)\in M\times\ind$}.$$
\end{lemma}
\begin{proof} 
Let us denote by $\kappa$ and $C$ the constants provided by Proposition \ref{prop a priori Lip} for $c:=c(\HH)$. Pick $\mu\in\Mis_0$.  
For $\ww\in{\mathcal F} $, we have 
\[
\min_i \min_M w_i\leqslant \int_{TM\times\ind} w_i(y)\,\dd\mu(y,v,i)\leqslant 0.
\]
Let $j\in\ind$ such that $\min_M w_j=\min_i\min_M w_i$. Since $w_j$ is $\kappa$-Lipschitz, we infer 
\[
\max_M w_j\leqslant \max_M w_j-\min_M w_j \leqslant \kappa\operatorname{diam}(M)<+\infty
\]
On the other hand, for $i\not=j$ we have  \ $w_i\leqslant w_j+\|w_i-w_j\|_\infty\leqslant \kappa\D{diam}(M)+C$ \ in $M$. 
 \end{proof}
Therefore we can define  $\uu^0:M\to \R^m$ by
\begin{equation}\label{first def u_0}
u^0_i(x):=\sup_{\ww\in{\mathcal F}  }w_i(x)\qquad\hbox{for every $(x,i)\in M\times\ind$}.
\end{equation}
As the supremum of an equi--Lipschitz family of critical subsolutions, we get that $\uu^0$ is  Lipschitz continuous and a critical subsolution as well, see \cite[Proposition 1.6]{DavZav14}. As a consequence of our convergence result, we will obtain in the end that $\uu^0$ is a critical solution belonging to ${\mathcal F} $.

We proceed by studying the asymptotic behavior of the discounted solutions $\uu^\lambda$  as $\lambda\to 0^+$ and the relation with $\uu^0$.  Let us denote by 
\[
\sol:=\left\{\uu\in\vLip\mid \, \uu=\lim_{k\to\ +\infty}\uu^{\lambda_k}\ \hbox{for some sequence $\lambda_k\to 0$}  \right\}.
\]
Note that any function in $\sol$ is a solution to the critical system \eqref{critical system} by the stability of the notion of viscosity solution.\smallskip

We begin with the following result:

\begin{prop}\label{prop constraint}
Let $\uu\in \sol$. Then
\begin{equation*}   %\label{constraint}
 \int_{ TM\times\ind} u_i(x)\, \dd \mu(x,v,i) \leqslant 0\qquad\hbox{for every $\mu\in\Mis_0$}.
\end{equation*}
In particular, $\uu\leqslant \uu^0$.
\end{prop}

\begin{proof}
Fix $\mu\in\Mis_0$. The assertion will be a direct consequence of the following fact: 
\begin{equation}\label{eq prima disuguaglianza}
\int_{ TM\times\ind} u^{\lambda}_i(x)\, \dd \mu(x,v,i) \leqslant 0\qquad\hbox{for every $\lambda>0$}.
\end{equation}
Indeed, let us fix $\lambda>0$. Regularizing $\uu^\lambda$ by convolution, we find a sequence of smooth functions $\ww^n:M\to\R^m$ 
such that $\ww^n \ucv \uu^\lambda$ and 
$$(B+\lambda \Id)\ww^n(x) +\HH\big(x, D \ww^n(x)\big) \leqslant \left(c(\HH)+\frac1n\right)\1\qquad\hbox{for every $x\in M$}.$$
By integrating this inequality with respect to $\mu$ and by using Fenchel's inequality we get
\begin{align*}
c(\HH)+\frac1n &\geqslant    \int_{ TM\times\ind}\Big((B+\lambda \Id)\ww^n(x)\Big)_i +H_i(x, D w^n_i)\, \dd \mu \\
&\geqslant  \int_{ TM\times\ind} \Big((B+\lambda \Id)\ww^n(x)^\eps\Big)_i +\langle D  w^n_i(x), v\rangle -L_i(x,v) \,\dd \mu \\
&=c(\HH)+ \int_{ TM\times\ind}\lambda w^n_i \,\dd \mu,
\end{align*}
where for the last equality we have used the fact that $\mu$ is closed and minimizing. The inequality \eqref{eq prima disuguaglianza} 
follows after sending $n\to +\infty$  and dividing by $\lambda>0$.
\end{proof}

The next (and final) step is to show that $\uu\geqslant \uu^0$ in $M$ whenever $\uu\in\sol$. This will be obtained by defining a special family of probability 
measures on $TM\times\ind$ for the discounted systems \eqref{system discounted critical}. The construction is the following: fix $(y,\ell)\in M\times\ind$ and, for every $\lambda>0$, let $\eta^\lambda:\Omega\to\curves$ be an admissible curve starting at $y$ that realizes the infimum in \eqref{discounted representation}.
We define a probability measure $\mu^\lambda_y$ on $TM\times\ind$ by setting 
\begin{equation}\label{def discounted measure}
\int_{ TM\times\ind} \ff\,\dd\mu^\lambda_y := \lambda\EE_\ell\left[ \int_0^{+\infty}\e^{-\lambda s} f_{\omega(s)}(\eta^\lambda(s,\omega),-\dot\eta^\lambda(s,\omega)\big)\,\dd s\right]\medskip
\end{equation}
for every $\ff\in\vCc$. The following holds:

\begin{prop}\label{prop discounted measures} The measures 
$\{\mu_y^\lambda\,|\,\lambda>0\}$ defined above are probability measures on $TM\times\ind$,
whose supports are all contained in a common compact subset of $TM\times\ind$. In particular, they are relatively compact in the space of 
probability measures on $TM\times\ind$ with respect to the weak convergence. Furthermore, if $\left(\mu_y^{\lambda_{n}}\right)_n$ is weakly converging to $\mu_y$ for some sequence $\lambda_n\to 0^+$, then $\mu_y$ is a minimizing Mather measure.
\end{prop}

\begin{proof} 
According to Theorem \ref{teo minimal discounted curve}, there exists a constant $\kappa^*$ such that \ $\|\dot\eta^\lambda(\cdot,\omega)\|_\infty\leqslant k^*$ \  for every $\omega\in\Omega$ and $\lambda>0$. Set $K:=\{(x,v)\in TM\mid  |v|\leqslant \kappa^*\}$. 
Then the measures $\mu_y^\lambda$ are all supported in the compact set $K\times\ind$ and are probability measures, as it can be easily checked by their definition. 
 This readily implies the asserted relative compactness of $\{\mu_y^\lambda\,|\,\lambda>0\}$. Let now assume that  $\left(\mu_y^{\lambda_n}\right)_n$ is weakly converging to $\mu_y$ for some $\lambda_n\to 0$. Then $\mu_y$ is a probability measure with support in $K\times\ind$, in particular it satisfies item (i) in Definition \ref{def closed measure}. 
 Moreover, if $\pphi\in\vCuno$, by Dynkin's formula applied to the function $\gb(t,x):=\e^{-\lambda t}\pphi(x)$, see Theorem \ref{Dynkin's formula}, we get  
 \begin{eqnarray*}
&&\!\!\!\!\!\!\!\!\!  \EE_\ell\left[\int_0^{+\infty}\!\!\! \e^{-\lambda s}\left( \langle D\phi_{\omega(s)}\big(\eta^\lambda(s,\omega)\big),-\dot\eta^\lambda(s,\omega)\rangle 
 +\big(B\pphi\big)_{\omega(s)}\big(\eta^\lambda(s,\omega)\big) +\lambda\phi_{\omega(s)}\big(\eta^\lambda(s,\omega)\big)  \right) \dd s \right]\\
 &&\qquad=\phi_\ell(y),
 \end{eqnarray*}
yielding
\[
\displaystyle\int_{ TM\times\ind } \big(B\pphi(x)\big)_i + \langle D \phi_i(x), v\rangle\  \dd \mu^\lambda_y 
=
\lambda\phi_\ell(y)-\lambda\int_{TM\times\ind} \phi_i\,\dd\mu^\lambda_y.
\]
By setting $\lambda:=\lambda_n$ in the previous equality and sending $n\to +\infty$ we infer 
\[
\displaystyle\int_{ TM\times\ind } \big(B\pphi(x)\big)_i + \langle D \phi_i(x), v\rangle\  \dd \mu_y=0,
\]
thus proving that $\mu_y$ is closed. 

To prove that $\mu_y$ is minimizing, we recall that, by definition,  
\begin{align*}
\lambda u^{\lambda}_{\ell}(y) & = \int_{ TM\times\ind} \big(L_i(x,v)+c(\HH)\big)\,\dd\mu^\lambda_y \qquad\hbox{for every $\lambda>0$.}
\end{align*}
The assertion follows by setting $\lambda:=\lambda_n$ and sending $n\to +\infty$. 
\end{proof}

We proceed by proving a lemma that will be crucial for the proof of Theorem \ref{intro main teo}. 

\begin{lemma}\label{ineq prim}
Let $\ww$ be any critical subsolution. For every $\lambda>0$ and $(y,\ell)\in M\times\ind$ we have 
\begin{equation}\label{eq pre-final}
u^\lambda_\ell(y)\geqslant w_\ell(y)-\int_{TM\times\ind} w_i\, \dd {\mu}^{\lambda}_y,
\end{equation}
where $\mu^\lambda_y$ is the probability measure defined by \eqref{def discounted measure}.
\end{lemma}

\begin{proof}
 Let $\ww$ be a critical subsolution. By convolution with a regularizing kernel, we construct a family of smooth function $\ww^n:M\to\R^m$ 
 uniformly converging $\ww$ such that 
$$B\ww^n(x) +\HH\big(x, D \ww^n(x)\big) \leqslant \left(c(\HH)+\frac1n\right)\1\qquad\hbox{for every $x\in M$}.$$
Starting again from the definition of $\uu^\lambda$ and by exploiting Fenchel's inequality we obtain
\begin{align*}
u^{\lambda}_\ell(y) &   =   \frac1\lambda\int_{ TM\times\ind} \big(L_i(x,v)+c(\HH)\big)\,\dd\mu^\lambda_y  \\
&\geqslant \frac1\lambda\int_{ TM\times\ind} \left(\langle D  w^n_{i}(x), v\rangle - H_i\big(x, D w^n_{i}(x)\right)+c(\HH)\big)\,\dd\mu^\lambda_y \\
&\geqslant \frac1\lambda\int_{ TM\times\ind} \left(\langle D  w^n_{i}(x), v\rangle +\big(B\ww^n(x)\big)_i -\frac1n\right)\, \dd\mu^\lambda_y.
\end{align*}
Using the definition of $\mu^\lambda_y$ and  Dynkin's formula with $\gb(t,x)=\e^{-\lambda t}\ww^n(x)$, see Theorem \ref{Dynkin's formula}, we get
\begin{eqnarray*}
u^{\lambda}_\ell(y) &\geqslant&  \EE_\ell\left[ \int_0^{+\infty}\e^{-\lambda s}  
\Big(\big\langle D  w^n_{\omega(s)}\big(\eta^\lambda(s,\omega)\big),-\dot\eta^\lambda(s,\omega)\big\rangle +\big(B\ww^n\big)_{\om(s)}\big(\eta^\lambda(s,\omega)\big) -\frac1n\Big)
\,\dd s\right]\\
&=&  w^n_\ell(y) +\EE_\ell\left[ \int_0^{+\infty}- \lambda \e^{-\lambda s} w^n_{\om(s)}\big(\eta^\lambda(s,\omega)\big) -\frac{\e^{-\lambda s}}{n}
\,\dd s\right] \\
&= &  w^n_\ell(y)-\int_{ TM\times\ind}w^n_{i} \,\dd\mu^\lambda_y - \frac{1}{\lambda n}.
\end{eqnarray*}
The desired inequality follows by sending $n\to +\infty$. 
\end{proof}

We have now all the ingredients to prove our main result. 

\begin{proof}[Proof of Theorem \ref{intro main teo}]
Let $\uu\in \sol$. By Proposition \ref{prop constraint}, we already know that $\uu\leqslant \uu^0$. Let us prove the opposite inequality. 
By definition, there exists a sequence $\lambda_n\to 0^+$ such that $\uu^{\lambda_n}\ucv \uu$ as $n\to +\infty$. Pick $\ww\in \F$ and 
fix $(y,\ell)\in  TM\times\ind$. By setting $\lambda:=\lambda_n$ in \eqref{eq pre-final} and by sending $n\to +\infty$, we infer, thanks to Proposition 
\ref{prop discounted measures}, that there exists a Mather measure $\mu_y\in\Mis_0$ such that 
 $$
 u_\ell(y)  \geqslant w_\ell(y) - \int_{ TM\times\ind} w_i\, \dd\mu_y\geqslant w_\ell(y),
 $$
 where, for the last inequality, we have used the fact that $\ww\in\F$. 
 As this is true for any $\ww\in\F$ and arbitrary $(y,\ell)\in M\times\ind$, we infer that $\uu\geqslant \uu^0$. This concludes the proof.
\end{proof}

\bibliography{weakly}
\bibliographystyle{siam}

\end{document}